\documentclass[10pt,reqno]{amsart}

\newcommand\version{September 5, 2011}

\usepackage[ansinew]{inputenc}

\usepackage{textcomp}

\usepackage{cite}

\usepackage{amsmath,amsfonts,amsthm,amssymb,amsxtra}

\usepackage[usenames,dvipsnames]{color}

\usepackage{bbm}
\usepackage{fancyhdr}
\usepackage[scanall]{psfrag}
\usepackage{graphicx}
\usepackage{ifthen}
\usepackage{pstricks,pst-node,pst-plot}

\usepackage{wasysym}

\newtheorem{theorem}{Theorem}[section]
\newtheorem{proposition}[theorem]{Proposition}
\newtheorem{lemma}[theorem]{Lemma}
\newtheorem{corollary}[theorem]{Corollary}

\theoremstyle{definition}

\newtheorem{definition}[theorem]{Definition}

\theoremstyle{remark}

\newtheorem{remark}[theorem]{Remark}

\numberwithin{equation}{section}

\newcommand{\C}{\mathbb{C}}

\renewcommand{\epsilon}{\varepsilon}

\newcommand{\N}{\mathbb{N}}

\newcommand{\R}{\mathbb{R}}

\newcommand{\var}{\varphi(x,\zeta)}
\newcommand{\tet}{\theta(x,\zeta)}

\DeclareMathOperator{\im}{Im}

\DeclareMathOperator{\re}{Re}

\DeclareMathOperator{\sgn}{sgn}
\DeclareMathOperator{\Tr}{Tr}
\DeclareMathOperator{\tr}{Tr}

\begin{document}

\title[Trace Formulae--- \version]{Trace formulas for Schr\"odinger operators on the half-line}

\author{Semra Demirel and Muhammad Usman}
\address{Semra Demirel, University of Stuttgart, Department of Mathematics, Pfaffenwaldring 57, D-70569 Stuttgart}
\email{Semra.Demirel@mathematik.uni-stuttgart.de}
\address{Muhammad Usman, Imperial College London, Department of Mathematics, London;
COMSATS Institute of Information Technology, Park Road, Chak Shahzad, Islamabad, Pakistan}
\email{m.usman08@imperial.ac.uk}

\begin{abstract}
We study the scattering problem for the Schr\"odinger equation on the half-line with Robin boundary condition at the origin. We derive an expression for the trace of the difference of the perturbed and unperturbed resolvent in terms of a Wronskian. This leads to a representation for the perturbation determinant and to trace identities of Buslaev-Faddeev type.
\end{abstract}

\maketitle

\section{Introduction}
Let $H$ be the self-adjoint operator in $L_2[0,\infty)$ defined by

\begin{equation}\label{erstegl}
H=H_0 + V(x),\quad H_0= -\frac{d^2}{dx^2}, \quad u'(0) = \gamma u(0),
\end{equation}
where $\gamma \in \R$. The potential $V$ is real-valued and goes to zero at infinity (in some averaged sense). Then $H$ has continuous spectrum on the positive semiaxis and discrete negative spectrum, consisting of eigenvalues $\{\lambda_j\}$.  If $V$ decays fast enough, then there are only finitely many negative eigenvalues.

The Hamiltonian $H$ describes a one-dimensional particle restricted to the positive semiaxis. The parameter $\gamma$ describes the strength of the interaction of the particle with the boundary. Negative $\gamma$ corresponds to an attractive interaction and positive $\gamma$ to a repulsive one.

In this paper we derive trace formulas for the negative eigenvalues of $H$. Formulas of this type first appeared in 1953 in the paper of Gel'fand and Levitan, \cite{GL}, where some identities for the eigenvalues of a regular Sturm-Liouville operator were obtained. Later, also Diki{\u\i} studied similar formulas, see \cite{D}. The next important contribution in this direction was made by Buslaev and Faddeev \cite{BF} in 1960. They studied the singular Sturm-Liouville operator on the half-line with Dirichlet boundary condition at the origin. Under some assumptions on the short range potential (i.e. integrable on $(0, \infty)$ with finite first moment), they proved a series of trace identities. The second one in this series states that
\begin{equation}\label{bf}
\sum_{j=1}^N |\lambda_j| - \frac{2}{\pi} \int_0^{\infty} \left( \tilde{\eta} (k) - \frac{1}{2k} \int_0^{\infty} V(x) \,dx \right) k  \,dk = \frac{1}{4} V(0),
\end{equation}
where $\tilde{\eta} (k)$ is the so-called limit phase and has a scattering theoretical nature. A more precise definition will be given later. This result was extended in 1997 by Rybkin to long-range potentials (nonintegrable on $(0,\infty)$), \cite{R,BR}. Analog formulas for charged particles were obtained already in 1972 by Yafaev \cite{Ya}. 

Trace formulas for the whole line Schr\"odinger operator as well as their generalizations to the multi-dimensional case have already been studied extensively (see, e.g., the surveys \cite{B,BY,K}) and have great importance in applications. Numerous papers are devoted to the subject of inverse spectral problems for Schr\"odinger operators, where these formulas turn out to be very useful, see e.g. \cite{DT,GH,GHSZ,A,BT} and references therein. Another consequence are  the well-known Lieb-Thirring inequalities, which in dimension one follow from the third Faddeev-Zakharov trace formula, see \cite{FZ} and \cite{LT}. This formula was extended in \cite{LW} by Laptev and Weidl to systems of Schr\"odinger operators, which leads to sharp Lieb-Thirring inequalities in all dimensions.

Our goal is to prove the analog of the Buslaev-Faddeev trace formulas for the half-line Schr\"odinger operator with Robin boundary conditions \eqref{erstegl}. Thereby, we follow Yafaev's book "Mathematical Scattering Theory, Analytic Theory"  \cite{Y}, which contains complete proofs in the case of Dirichlet boundary condition. We aim to point out the differences arising from the Robin boundary conditions and to give an interpretation for them.

The outline of this paper is as follows. We consider the differential equation
\begin{equation}\label{eq}
-u'' + V(x) u = z u,\quad \ z=\zeta^2,
\end{equation}
where $\zeta \in \C$ and $x >0$. We are concerned with two particular solutions of this equation, the regular solution $\varphi$ and the Jost solution $\theta$. The first one is characterized by the conditions
\begin{equation}\label{bc2}
\varphi(0, \zeta) = 1, \ \ \varphi'(0, \zeta) = \gamma,
\end{equation}
and the latter one by the asymptotics $\theta(x,\zeta) \sim e^{i\zeta x}$ as $x \to \infty$.

In section 2 we prove existence and uniqueness of the regular solution. The corresponding properties of the Jost solution are well-known. Further, we introduce a quantity $w(\zeta)$, which we call the \emph{Jost function}. We emphasize that this function depends on $\gamma$ and does \textit{not} coincide with what is called the Jost function in the Dirichlet case. More precisely ,  $w(\zeta)$ is defined as the Wronskian of the regular solution and the Jost solution of \eqref{eq}. It turns out that
$$
w(\zeta) = \gamma \theta(0,\zeta) - \theta'(0,\zeta).
$$

Section 3 contains our first main result. Denoting the resolvents of the unperturbed and perturbed operators by $R_0(z)= (H_0-z)^{-1}$ and $R(z) = (H-z)^{-1}$, respectively, we derive an expression for $\tr (R(z) -R_0(z))$ in terms of the Jost function.

\begin{theorem}
Assume that $\int_0^{\infty} |V(x)| \,dx < \infty$. Then
\begin{equation}
\Tr ( R_0(z) - R(z)) = \frac{1}{2\zeta}\left( \frac{\dot{w}(\zeta)}{w(\zeta)} + \frac{i}{\gamma - i\zeta} \right), \quad \zeta = z^{1/2},\ \im \zeta > 0.
\end{equation}
\end{theorem}
\noindent From this relation we get a representation for the perturbation determinant in terms of $w(\zeta)$.

Section 4 deals with the asymptotic expansion of the perturbation determinant, which we shall use to derive trace identities in Section 5.
For complex numbers $s$, we define the function
$$
M_s(\gamma) := \begin{cases} (-\gamma)^{2s} \ \  &\mbox{if}\  \gamma < 0, \\ 0\ \quad  &\mbox{if}\  \gamma \geq 0. \end{cases} 
$$
Under some regularity and decay assumptions on the potential $V$ we prove infinitely many trace identities. The analogous to \eqref{bf} will be given by
\begin{equation}\label{eq:trafointro}
\sum_{j=1}^N |\lambda_j| - M_1(\gamma) - \frac{2}{\pi}  \int_0^{\infty} \left( \eta(k) -\frac{1}{2k} \int_0^{\infty}V(x)\,dx \right) k \,dk = -\frac14 V(0),
\end{equation}
where $\eta(k)$ is now the corresponding limit phase for the Robin boundary problem. We recall that if $\gamma \geq 0$, then $H_0$ has purely absolutely continuous spectrum $[0, \infty)$. If $\gamma <0$, then $H_0$ has a simple negative eigenvalue $-\gamma^2$ and purely absolutely continuous spectrum on $[0,\infty)$. Hence the first two terms on the left-hand side of \eqref{eq:trafointro}, $\sum_{j=1}^N |\lambda_j| - M_1(\gamma)$, correspond to the shift of the discrete spectrum between $H$ and $H_0$. Similarly, the last term on the left-hand side corresponds to the shift of the absolutely continuous spectrum. The trace formula \eqref{eq:trafointro} and its higher order analogs proved below relate this shift of the spectrum to the potential $V$.

Finally, in Theorem \ref{levinson} we prove a trace formula of order zero. Namely,  the so-called \textit{Levinson formula} for the Schr\"odinger operator $H$ with Robin boundary condition.

{\bf Acknowledgments}. The authors are grateful to A. Laptev, T. Weidl and D. Yafaev for helpful inputs. Many thanks to Rupert Frank for informative discussions and references. 


\section{The regular solution and the Jost solution}
In this section, we prove existence and uniqueness of the regular solution and recall some elementary results on the Jost solution. The $\gamma$-dependent Jost function is studied.

\subsection{The associated Volterra equation and auxiliary estimates}
Existence and uniqueness of the regular solution of \eqref{eq} can be proved by using Volterra integral equations. For different boundary conditions, equation \eqref{eq} is associated with different Volterra integral equations. 

\begin{lemma}\label{regular}
Let $V \in L_1^{(loc)}([0, \infty))$ and consider equation \eqref{eq} on functions $\varphi \in C^1([0, \infty))$, such that $\varphi'$ is absolutely continuous. Then \eqref{eq} with boundary conditions \eqref{bc2} is equivalent to the Volterra equation
\begin{equation}\label{ieq}
\varphi(x, \zeta) = \cos(\zeta x)+ \frac{\gamma}{\zeta} \sin(\zeta x) + \frac{1}{\zeta} \int_0^x \sin(\zeta(x-y)) V(y) \varphi(y, \zeta) \,dy ,
\end{equation}
considered on locally bounded functions $\varphi$.
\end{lemma}

\begin{proof}
Suppose that equation \eqref{eq} holds for $\varphi$. Then the equality
$$
\int_0^x \zeta^{-1} \sin(\zeta(x-y))V(y) \varphi(y,\zeta) \,dy = \int_0^x \zeta^{-1} \sin(\zeta(x-y))\left(\varphi''(y,\zeta) + \zeta^2 \varphi(y,\zeta)\right) \,dy
$$
is true.
We integrate the right-hand side twice by parts. Taking into account boundary conditions \eqref{bc2}, we see that the right-hand side equals $ \varphi(x,\zeta) - \cos(\zeta x)- \frac{\gamma}{\zeta} \sin(\zeta x).$ Thus equation \eqref{ieq} follows.
Conversely, assume that equation \eqref{ieq} holds. Then $\varphi \in C_{loc}^1([0,\infty))$ and
\begin{equation}\label{abl} 
\varphi'(x,\zeta) = - \zeta \sin(\zeta x) + \gamma \cos(\zeta x) + \int_0^x  \cos(\zeta(x-y))V(y) \varphi(y,\zeta) \,dy.
\end{equation}
Therefore $\varphi'$ is absolutely continuous and 
\begin{equation}\label{abl2} 
\varphi''(x,\zeta) = -\zeta^2 \cos(\zeta x) - \gamma \zeta \sin(\zeta x) - \zeta \int_0^x  \sin(\zeta(x-y))V(y) \varphi(y,\zeta) \,dy.
\end{equation}
Comparing \eqref{abl2} with \eqref{ieq}, we obtain equation \eqref{eq}. Inserting $x=0$ in \eqref{ieq} and \eqref{abl} for $\varphi(x,\zeta)$ and $\varphi'(x,\zeta)$, we see that boundary conditions \eqref{bc2} are fullfilled.
\end{proof}

In the following Lemma it is proved that the regular solution $\var$  of \eqref{eq} with boundary conditions \eqref{bc2} exists uniquely. For the case of Dirichlet boundary condition, this result was represented e.g. by Yafaev in \cite{Y}. 

\begin{lemma}\label{regularsolution}
Let $ V \in L_1^{(loc)} ([0, \infty)).$ Then for all $\zeta \in \C$, equation \eqref{eq} has a unique solution $\varphi(x,\zeta)$ satisfying \eqref{bc2}. For any fixed $x\geq 0, \ \var = \varphi(x,-\zeta)$ is an entire function of the variable $z=\zeta^2$. Moreover, for $\gamma \neq 0$ we have the estimate
\begin{equation}\label{unique}
\left | \varphi(x,\zeta) - \cos (\zeta x) -\frac{\gamma}{\zeta} \sin (\zeta x) \right | \leq \tilde{c} | \gamma| x e^{| \im \zeta | x} \left( \exp\left(\frac{c \int_0^x |V(y)| (1+| \gamma| y) \,dy}{| \gamma|} \right) -1 \right).
\end{equation}
If $\gamma = 0$, then the estimate
\begin{equation}\label{unique2}
\left | \varphi(x,\zeta) - \cos (\zeta x) \right | \leq \tilde{c} e^{| \im \zeta | x} \left( \exp\left( c x \int_0^x |V(y)| \,dy \right) -1 \right)
\end{equation}
holds.
\end{lemma}

\begin{proof}
We construct a solution of integral equation \eqref{ieq}, which by Lemma \ref{regular} is equivalent to the solution of \eqref{eq}. 
Set $\varphi_0(x,\zeta) = \cos (\zeta x) + \frac{\gamma}{\zeta} \sin (\zeta x)$,
\begin{equation}\label{eins}
\varphi_{n+1} (x,\zeta) = \int_0^x \zeta^{-1} \sin(\zeta(x-y))V(y) \varphi_n(y,\zeta) \,dy  ,\quad n \geq 0.
\end{equation}
Inductively one shows that all $\varphi_n(x,\zeta)$ are entire functions of $\zeta^2$. First, we consider the case when $\gamma \neq 0$. Using the estimates

\begin{equation}\label{obvious}
\left | \frac{ \sin(\zeta (x-y))}{\zeta} \right | \leq c |x-y| e^ {| \im \zeta | (x-y)} \quad \mbox{and} \quad  \left | \cos(\zeta x) + \frac{\gamma}{\zeta} \sin(\zeta x) \right | \leq \tilde{c}   e^{| \im \zeta | x} ( 1+ |\gamma| x) ,
\end{equation}
we obtain
\begin{equation}\nonumber
\left | \varphi_1 (x, \zeta ) \right | \leq c \tilde{c} x e^{| \im \zeta | x}   \int_0^x |V(y)| (1+\left | \gamma \right | y) \,dy.
\end{equation}
Successively, we have
\begin{equation}\label{zwei}
\left | \varphi_n (x, \zeta ) \right | \leq \frac{c^{n}\tilde{c}}{n! \left | \gamma \right |^{ n-1}} x e^{| \im \zeta | x} \left( \int_0^x |V(y)| (1+\left | \gamma \right | y) \,dy \right)^n,
\end{equation}
which follows by an induction argument. Indeed, it follows from \eqref{eins}, \eqref{obvious} and \eqref{zwei}, that
\begin{eqnarray}\label{y}
\left | \varphi_{n+1} (x, \zeta ) \right | \leq \frac{c^{n+1}\tilde{c}}{n! \left | \gamma \right |^{ n-1}} e^{| \im \zeta | x} \int_0^x (x-y)y |V(y)| \left( \int_0^y |V(t)| (1+ \left | \gamma \right | t)  \,dt \right)^n \,dy.
\end{eqnarray}
As $\left | \gamma\right  |> 0$, we have $y \leq \left | \gamma \right |^{-1} (1+ |\gamma| y)$. Thus, the right-hand side in equation \eqref{y} is bounded by
\begin{eqnarray}\nonumber
\frac{c^{n+1}\tilde{c}}{ n! |\gamma|^{ n}} e^{| \im \zeta | x} \int_0^x (1+|\gamma| y) (x-y) |V(y)| \left( \int_0^y |V(t)| (1+|\gamma| t)  \,dt \right)^n \,dy, 
\end{eqnarray}
which is the same as 
\begin{eqnarray}\nonumber
\frac{c^{n+1}\tilde{c}}{n! | \gamma|^{ n}} e^{| \im \zeta | x} \int_0^x \frac{(x-y)}{n+1} \frac{d}{dy} \left( \int_0^y |V(t)| (1+|\gamma| t)  \,dt \right)^{n+1} \,dy.
\end{eqnarray}
Finally, the last term is bounded from above by
\begin{equation}\nonumber
\frac{c^{n+1}\tilde{c}}{(n+1)! | \gamma|^{ n}} e^{| \im \zeta | x} x  \left( \int_0^x |V(y)| (1+|\gamma| y)  \,dy \right)^{n+1} .
\end{equation}
Thus, the limit 
\begin{equation}\label{limit3}
\varphi(x,\zeta):= \lim_{N \to \infty} \sum_{n=0}^N \varphi_n (x,\zeta)
\end{equation} 
exists uniformly for bounded $\zeta$, $x$ and $|\gamma| >0$. Putting together definitions \eqref{eins} and \eqref{limit3}, we see that
$$
\sum_{n=0}^N \varphi_n(x,\zeta) =  \cos (\zeta x) + \frac{\gamma}{\zeta} \sin (\zeta x) - \varphi_{N+1}(x,\zeta) +  \int_0^x \zeta^{-1} \sin(\zeta(x-y))V(y) \left( \sum_{n=0}^N \varphi_n(y,\zeta) \right) \,dy.
$$ 
From this equation, we obtain in the limit $N \to \infty$ equation \eqref{ieq}. To prove estimate \eqref{unique}, we consider
\begin{equation}\label{summe}
\left| \varphi(x,\zeta) - \cos(\zeta x) -\frac{\gamma}{\zeta} \sin(\zeta x)\right| = \left |  \lim_{N \to \infty} \sum_{n=1}^N \varphi_n (x,\zeta) \right|.
\end{equation}
Because of \eqref{zwei}, the right-hand side in \eqref{summe} is bounded from above by

\begin{eqnarray}\nonumber
\tilde{c} |\gamma| x e^{ | \im \zeta |x} \sum_{n=1}^{\infty} \frac{1}{n!}\left(\frac{c \int_0^x |V(y)| (1+|\gamma| y) \,dy}{|\gamma|} \right)^n = \tilde{c} |\gamma| x e^{ | \im \zeta |x}  \left( \exp\left(\frac{c \int_0^x |V(y)| (1+|\gamma| y) \,dy}{|\gamma|} \right) -1\right).
\end{eqnarray}
If $\gamma =0$, then we use the same estimates \eqref{obvious} and get successively,
\begin{equation}\nonumber
| \varphi_n(x,\zeta) | \leq \frac{c^n \tilde{c}x^n}{n!} e^{|\im \zeta| x} \left( \int_0^x |V(y)|  \,dy \right)^n.
\end{equation}
From this estimate it follows, that
$$
\left| \varphi(x,\zeta) - \cos(\zeta x) \right| \leq \tilde{c}  e^{| \im \zeta |x} \sum_{n=1}^{\infty} \frac{1}{n!} \left( c x  \int_0^x |V(y)|  \,dy \right)^n = \tilde{c}  e^{| \im \zeta |x} \left( \exp \left(cx \int_0^{x} |V(y)| \,dy \right) -1 \right),
$$
which proves estimate \eqref{unique2}.
The uniqueness of a bounded solution of equation \eqref{ieq} can be proved by contradiction. Suppose that $\varphi_1$ and $\varphi_2$ are two different solutions of equation \eqref{ieq}. Then $\varphi_1- \varphi_2$ satisfies the corresponding homogeneous equation and is bounded for an arbitrary $n$, by the right-hand side of \eqref{zwei} and hence is zero. Therefore $\varphi_1= \varphi_2$.
\end{proof}

\subsection{The Jost solution and the Jost function}

The so-called Jost solution, which was first studied by Jost, is important in scattering theory. This solution of equation \eqref{eq} is characterized by the asymptotics $\tet \sim e^{i \zeta x}$ as $x\to \infty$.
It is proved, e.g., in \cite{Y}, that under the assumption 
\begin{equation}\label{cond}
\int_0^{\infty} |V(x)| \,dx < \infty,
\end{equation}
equation \eqref{eq} has for all $\zeta \neq 0,\ \im \zeta >0$, a unique solution $\tet$ satisfying as $x \to \infty$ the conditions 
\begin{equation}\label{asympto}
\tet = e^{i \zeta x} (1 + o(1)),\quad \theta'(x,\zeta) = i\zeta e^{i \zeta x} (1+o(1)).
\end{equation}
For any fixed $ x \geq 0$, the function $\tet$ is analytic in $\zeta$ in the upper half-plane $\im \zeta > 0$ and continuous in $\zeta$ up to the real axis with a possible exception of the point $\zeta = 0$. Moreover, it satisfies the estimates
$$
| \tet - e^{i \zeta x}| \leq e^{- \im \zeta x} \left(\exp(|\zeta|^{-1} \int_x^{\infty} |V(y)| \,dy) -1\right)
$$ 
and consequently, for  $|\zeta| \geq c >0$,
\begin{equation}\label{yafaev}
| \tet - e^{i \zeta x}| \leq C |\zeta|^{-1} e^{- \im \zeta x} \int_x^{\infty} |V(y)| \,dy,
\end{equation}
where $C$ depends on $c$ and the value of the integral \eqref{cond} only. We will need an analog of estimate \eqref{yafaev} for the derivative of the Jost solution.
\begin{lemma}\label{appendix}
Assume condition \eqref{cond} and  $\zeta \neq 0,\ \im \zeta >0$. Then for the derivative of the solution $\tet$ with asymptotics \eqref{asympto}, the following estimate holds
\begin{equation}\label{eigeneabschatzung}
| \theta'(x,\zeta) - i\zeta e^{i \zeta x}| \leq e^{- \im \zeta x} | \zeta| \left(\exp\left(|\zeta|^{-1} \int_x^{\infty} |V(y)| \,dy\right) -1\right).
\end{equation}
Moreover, for $|\zeta| \geq k >0$, we have

\begin{equation}\label{schluss}
| \theta'(x,\zeta) - i\zeta e^{i \zeta x}| \leq K e^{- \im \zeta x} \int_x^{\infty} |V(y)| \,dy,
\end{equation}
where $K$ depends only on $k$ and the value of the integral \eqref{cond}.
\end{lemma}

The proof of this Lemma follows closely the arguments of \cite{Y}. For the sake of completeness, we provide the necessary modifications in the appendix.

Next, we study some properties of the $\gamma$-dependent Jost function.
Below we suppose that condition \eqref{cond} is satisfied and that $\im \zeta \geq 0$.
\begin{definition}
We denote by
\begin{equation}\label{jostfct}
w(\zeta) := \varphi'(x, \zeta) \tet - \theta'(x, \zeta) \var
\end{equation}
the Wronskian of the regular solution and the Jost solution of the Schr\"odinger equation \eqref{eq}. The Wronskian $w(\zeta)$ is called \textit{Jost function}. 
\end{definition}
Setting $x=0$ in \eqref{jostfct}, we see that
\begin{equation}\label{w}
w(\zeta) = \gamma \theta(0,\zeta) - \theta' (0, \zeta).
\end{equation}
This is the definition of $w(\zeta)$, that was used in the introduction. The Jost function $w(\zeta)$ is analytic in $\zeta$ in the upper halfplane $\im \zeta >0$ and is continuous in $\zeta$ up to the real axis, with a possible exception of the point $\zeta = 0$. Moreover, it follows from \eqref{yafaev} and \eqref{schluss} that
\begin{equation}\label{ruscho}
w(\zeta)=  - i \zeta  + O(1),\quad |\zeta | \to \infty, \quad \im \zeta \geq 0.
\end{equation}

\begin{remark}
Usually, in the literature the Wronskian $w_D$ of the Jost solution and the regular solution satisfying a \emph{Dirichlet} boundary condition is called Jost function. In our case of Robin boundary condition \eqref{bc2}, the Wronskian differs from the usual one and depends on $\gamma$. We emphasize, that for every $\gamma \in \R$, the function $w(\zeta)$ grows linearly in $\zeta$ as $|\zeta| \to \infty$, whereas in the Dirichlet case we have for the corresponding Jost function $w_D(\zeta) = 1+O(|\zeta|^{-1}), \ |\zeta| \to \infty$. 
\end{remark}

Our next goal is to give an integral representation for $w(\zeta)$.
\begin{lemma}\label{lemma_w}
For $\im \zeta \geq 0,\ \zeta \neq 0$, the following representation for the Jost function holds
\begin{equation}\label{represantation_w}
w(\zeta) = \gamma - i \zeta + \int_0^{\infty} e^{i\zeta y} V(y) \varphi(y,\zeta) \,dy.
\end{equation}
\end{lemma}

The proof of this Lemma relies on the following formula. For $\im \zeta >0$,
\begin{equation} \label{limit2}
\lim_{x \to \infty} e^{i\zeta x} \left( \varphi'(x,\zeta) -i \zeta \varphi(x,\zeta) \right) = w(\zeta).
\end{equation}
To show this, one can introduce, as in \cite{Y}, for all $\zeta$ with $\im \zeta>0$ a solution of equation \eqref{eq}, which is linearly independent of $\theta(x,\zeta)$. Set
$$
\tau(x,\zeta) = -2i\zeta \theta(x,\zeta) \int_{x_0}^x \theta(y,\zeta)^{-2} \,dy,\quad x \geq x_0,
$$
where $x_0=x_0(\zeta)$ is chosen such that $\theta(x,\zeta) \neq 0$ for all $x \geq x_0$. Then $\tau(x,\zeta)$ satisfies equation \eqref{eq} and according to \eqref{asympto},
$$
\tau(x,\zeta) = e^{-i\zeta x} (1+o(1)),\quad \tau'(x,\zeta) = -i\zeta e^{-i\zeta x} (1+o(1)),
$$ 
as $ x\to \infty$. Since $W\{ \theta(\zeta), \tau(\zeta) \} = 2i\zeta$, we find that 
\begin{equation} \label{varphi_tau}
\varphi(x,\zeta)= \frac{1}{2i\zeta} \left( (\gamma \tau(0,\zeta) -\tau'(0,\zeta)) \theta(x,\zeta) - (\gamma \theta(0,\zeta) - \theta'(0,\zeta))\tau(x,\zeta) \right).
\end{equation}
Equation \eqref{limit2} now follows from \eqref{varphi_tau}. Given \eqref{limit2}, we can prove Lemma \ref{lemma_w}.\\

\begin{proof}[Proof of Lemma \ref{lemma_w}]
The differential equation \eqref{eq} implies that
$$
\int_0^ x e^{i\zeta y} V(y) \varphi(y,\zeta) \,dy = \int_0^ x e^{i\zeta y} \varphi''(y,\zeta) \,dy + \zeta^2 \int_0^x e^{i\zeta y} \varphi(y,\zeta) \,dy .
$$
We integrate the first integral in the right-hand side twice by parts and get
$$
\int_0^x e^{i\zeta y} V(y) \varphi(y,\zeta) \,dy = e^{i\zeta x} \left( \varphi'(x,\zeta) -i \zeta \varphi(x,\zeta) \right) -\gamma +i \zeta.
$$
Passing to the limit $x \to \infty$ in the above equation and using \eqref{limit2}, we arrive at \eqref{represantation_w} for $\im \zeta > 0$. By continuity, \eqref{represantation_w} can be extended to the real axis.
\end{proof}

As the Jost solution of equation \eqref{eq} is unique, it follows that
\begin{equation}\label{regeln}
\theta(x,\zeta) = \overline{\theta(x,-\overline{\zeta})}, \ \  \theta'(x,\zeta) = \overline{\theta'(x,-\overline{\zeta})}\quad \mbox{and hence} \quad w(\zeta)= \overline{w(-\overline{\zeta})}.
\end{equation}
For real numbers $k>0$ both Jost solutions $\theta(x,k)$ and $\theta(x,-k)$ of the equation
\begin{equation}\label{realeq}
-u'' + V(x) u = k^2 u,\quad k>0,
\end{equation}
are correctly defined and their Wronskian $W\{ \theta(\cdot,k),\theta(\cdot, -k) \}$ equals $2ik$. Thus, they are linearly independent. In particular, we get from \eqref{regeln},
\begin{equation}\label{realw}
\theta(x,-k)= \overline{\theta(x,k)} \quad \mbox{and hence} \quad w(-k) = \overline{w(k)}.
\end{equation}
It is useful to express the regular solution in terms of the Jost solutions as follows,
\begin{eqnarray}\label{real}
\varphi(x,k) = \frac{1}{2ik}\left( \theta(x,k) w(-k) - \theta(x,-k)w(k) \right).
\end{eqnarray}
Indeed, it is easy to verify that the right-hand side of \eqref{real} satisfies equation \eqref{realeq} and conditions \eqref{bc2}.

Now, we introduce the limit amplitude and phase shift for real values of $k$.
\begin{definition}
Set
\begin{equation}\label{nurw}
w(k) = a(k) e^{i \eta(k)} (\gamma-ik), \quad a(k) = \frac{|w(k)|}{\sqrt{\gamma^2 + k^2}}.
\end{equation}
The functions $a(k)$ and $\eta(k)$ are called the limit amplitude and the limit phase, respectively.
\end{definition}
These functions determine the asymptotics of the regular solution of the Schr\"odinger equation as $x \to \infty$. Indeed, comparing \eqref{asympto} and \eqref{real}, we find
\begin{eqnarray}\nonumber
\varphi(x,k) & =& \frac{1}{2ik}\left( e^{ikx} w(-k) - e^{-ikx}w(k) \right) + o(1),\quad x\to \infty.
\end{eqnarray}
Furthermore,
$$
w(-k) = \overline{w(k)}= a(k) e^{-i \eta(k)} (\gamma + ik).
$$
Thus,
\begin{equation}\nonumber
\varphi(x,k) = a(k) \frac{1}{2ik} \left( (\gamma+ik) e^{i(kx-\eta(k))}  - (\gamma-ik) e^{-i(kx-\eta(k))} \right) + o(1), \quad x \to \infty. 
\end{equation}
This asymptotic behavior should be compared with the exact expression for the solution $\varphi_0(x,\zeta)$ of the equation $-\varphi'' = \zeta^2 \varphi$ satisfying the conditions \eqref{bc2}, namely,
$$
\varphi_0(x,\zeta) = (2i\zeta)^{-1} ((\gamma + i\zeta) e^{i\zeta x} - (\gamma-i\zeta)e^{-i\zeta x}).
$$

Finally, we note, that
\begin{equation}\label{nichtnull}
w(k) \neq 0 \quad \mbox{for all} \ k>0.
\end{equation}
Indeed, if there was a number $k$ such that $w(k)=0$, then it would follow from relations \eqref{realw} and \eqref{real}, that $\varphi(x,k)=0$ for all $x$.


\section{A Trace formula and the perturbation determinant}

We consider the Hamiltonian 
\begin{equation}\label{hamiltonian}
H= -\frac{d^2}{dx^2} + V(x), \quad V= \overline{V},
\end{equation}
with boundary condition \eqref{bc2} in the space $L_2(\R_+)$. More precisely, $H$ is defined through the quadratic form
$$
\int_0^{\infty} \left( |u'(x)|^2 +V(x) |u(x)|^2 \right) \,dx + \gamma |u(0)|^2
$$
with form domain $H^1(\R_+)$.
By $H_0= -\frac{d^2}{dx^2}$ we denote the free Hamiltonian with the same boundary condition \eqref{bc2} but with $V\equiv 0$. The resolvents of $H$ and $H_0$ are denoted by $R(z)$ and $R_0(z)$, respectively.

In this section, we derive an expression for $\tr \left(R(z)-R_0(z)\right)$ in terms of the Jost function. From this relation we get a representation for the perturbation determinant.

It is a well-known fact, that $R(z)$ can be constructed in terms of solutions $\varphi(x,\zeta)$ and $\tet$ of equation \eqref{eq} and their Wronskian \eqref{jostfct}.
Suppose, that \eqref{cond} holds. Then for all $z$ such that $\im z \neq 0$ and $w(\zeta) \neq 0$, the resolvent is the integral operator with kernel
\begin{equation}\label{reso}
R(x,y;z) = w^{-1}(\zeta) \var \theta(y,\zeta),\quad x\leq y,\ \zeta = z^{1/2},
\end{equation}
and $R(x,y;z) = R(y,x;z)$. Moreover, the estimate
$$
|R(x,y;z)| \leq c |w(\zeta)|^{-1} |\zeta|^{-1} \exp(- \im \zeta |x-y|)
$$
holds. We note, that in the particular case $V\equiv0$, the unperturbed resolvent $R_0(z)$ has the integral kernel
\begin{equation}\label{freekernel}
R_0(x,y;z) = R_0(y,x;z) = \frac{ ((\gamma +i\zeta)e^{i\zeta x} -(\gamma-i\zeta) e^{-i\zeta x}) e^{i \zeta y}}{2i\zeta (\gamma -i \zeta)},\quad x\leq y.
\end{equation}

The self-adjoint operator $H$ has discrete negative spectrum, which consists of negative eigenvalues $\lambda_j  = (i \kappa_j)^2,\ \kappa_j >0$, which possibly accumulate at zero. It is important to note, that the zeros of the function $w(\zeta)$ and the eigenvalues of $H$ are related as follows.

\begin{lemma}\label{zero}
Complex zeros of the function $w(\zeta)$ are simple and lie on the imaginary axis. Moreover, $w(\zeta)=0$ if and only if $\lambda = \zeta^2$ is a negative eigenvalue of the operator $H$.
\end{lemma}
\begin{proof}
First, assume that $w(\zeta)=0$ for $\im \zeta > 0$. Then the Jost function $\tet$ fullfills boundary conditions \eqref{bc2} and is in the space $L_2(\R_+)$ because of \eqref{asympto}. Thus $\tet$ is an eigenfunction of the operator $H$ corresponding to the eigenvalue $\lambda = \zeta^2$. Since $H$ is self-adjoint, it follows that $\lambda < 0$. Conversely, assume that $\lambda$ is an eigenvalue of $H$. Then its resolvent $R(z)$ has a pole in $\lambda$. Therefore, it follows from \eqref{reso} that $w(\zeta)=0$. As the resolvent of a self-adjoint operator has only simple poles, the zeros of $w(\zeta)$ are simple.
\end{proof}

\begin{remark}
It follows from the proporties of the regular solution and Jost solution, that the resolvent kernel \eqref{reso} is an analytic function in the upper half-plane $\im \zeta > 0$, except for simple poles at eigenvalues of $H$. In view of \eqref{nichtnull}, the resolvent kernel is a continuous function of $z$ up to the cut along $[0, \infty)$  with the possible exception of the point $z=0$. 
\end{remark}

\begin{proposition}
Assume condition \eqref{cond}, then
\begin{equation}\label{Tr}
\Tr ( R_0(z) - R(z)) = \frac{1}{2\zeta}\left( \frac{\dot{w}(\zeta)}{w(\zeta)} + \frac{i}{\gamma - i\zeta} \right), \quad \zeta = z^{1/2},\ \im \zeta > 0.
\end{equation}
\end{proposition}

\begin{proof}
Since $R-R_0$ is a trace class operator and kernels of the operators $R$ and $R_0$ are continuous functions, we have  
\begin{equation}
\Tr (R(z) - R_0(z)) = \lim_{x \to \infty} \int_{y=0}^{y=x} (R(y,y;z) - R_0(y,y;z)) \,dy.
\end{equation}
Using \eqref{freekernel}, we first compute
\begin{equation}\label{frei}
2 \zeta \int_{0}^{x} R_0(y,y;z) \,dy = \frac{1}{i(\gamma-i\zeta)} \left( \frac{\gamma + i \zeta}{2i\zeta} e^{2i\zeta x} -(\gamma -i\zeta)x-\frac{\gamma+i\zeta}{2i\zeta} \right).
\end{equation}
The following equation is true for any two arbitrary solutions of equation \eqref{eq}
\begin{equation}\label{useful}
2 \zeta \var \tet = (\varphi'(x,\zeta) \dot{\theta}(x,\zeta) - \var \dot{\theta}'(x,\zeta))'.
\end{equation}
Applying \eqref{useful}  to the regular solution $\varphi(x,\zeta)$ and the Jost solution $\tet$, we get
\begin{eqnarray}\label{echte}
2 \zeta w(\zeta)  \int_{0}^{x} R(y,y;z) \,dy = 2 \zeta  \int_{0}^{x} \varphi(y,\zeta) \theta(y,\zeta) \,dy = \left[ \varphi'(y,\zeta) \dot{\theta}(y,\zeta) - \varphi(y,\zeta) \dot{\theta}'(y,\zeta)\right]_0^x.
\end{eqnarray} 
Note, that the contribution of the right-hand side in \eqref{echte} for $y=0$ is
\begin{equation}\label{0}
\varphi'(0,\zeta) \dot{\theta}(0,\zeta) - \varphi(0,\zeta) \dot{\theta}'(0,\zeta) = \gamma \dot{\theta}(0,\zeta) -  \dot{\theta}'(0,\zeta) = \dot{w}(\zeta).
\end{equation}
Consider now the case of potentials of compact support. Then, for all $\zeta \in \C$ , we have $\tet= e^{i\zeta x}$ for sufficiently large $x$. Further, in this case we can generalize \eqref{real} to complex $\zeta$, namely
\begin{equation}\label{complex}
\varphi(x,\zeta) = \frac{1}{2i\zeta}\left( \theta(x,\zeta) w(-\zeta) - \theta(x,-\zeta) w(\zeta) \right).
\end{equation}
For large $x$ we have,
$$
\tet = e^{ix\zeta},\quad \theta'(x,\zeta)= i\zeta e^{ix\zeta},\quad \dot{\theta}(x,\zeta) = ix e^{ix\zeta}, \quad \dot{\theta}'(x,\zeta)=(i-x\zeta)e^{ix\zeta}.
$$
Taking into account that, since $\im \zeta >0$, the terms containing $e^{2ix\zeta}$ tend to zero as $x \to \infty$ and using \eqref{complex}, we get for sufficiently large $x$
\begin{equation}\label{x}
\varphi'(x,\zeta) \dot{\theta}(x,\zeta) - \varphi(x,\zeta) \dot{\theta}'(x,\zeta) = (ix + (2\zeta)^{-1})w(\zeta) + o(1).
\end{equation}
Combining \eqref{0} and \eqref{x}, we arrive at
\begin{equation}\label{nichtfrei}
2 \zeta \int_{0}^{x} R(y,y;z) \,dy = ix + \frac{1}{2\zeta} -  \frac{\dot{w}(\zeta)}{w(\zeta)} + o(1).
\end{equation}
Finally, we conclude from \eqref{frei} and \eqref{nichtfrei}, that
$$
\lim_{x \to \infty} \int_{0}^{x} (R_0(y,y;z) - R(y,y;z)) \,dy = \frac{1}{2 \zeta} \left( \frac{\dot{w}(\zeta)}{w(\zeta)} + \frac{i}{\gamma -i \zeta}\right) +o(1).
$$
This proves \eqref{Tr} for compactly supported potentials $V$. By density arguments (see \cite[Prop. 4.5.3]{Y}), based on the fact that $\sqrt{|V|} (H_0+\mathbbm{1})^{-1/2}$ is a Hilbert-Schmidt operator under condition \eqref{cond},  the result can be extended to all potentials $V$ satisfying this condition.
\end{proof}

We conclude this section by relating the Jost function to the perturbation determinant. Since $\sqrt{|V|} (H_0+\mathbbm{1})^{-1/2}$ is a Hilbert-Schmidt operator, the operator $\sqrt{V} R_0(z) \sqrt{|V|}$ is a trace class operator (here $\sqrt{V}:= (\sgn V) \sqrt{|V|}$) and therefore the (modified) perturbation determinant
$$
D(z) :=  \mbox{det} (\mathbbm{1} +\sqrt{V} R_0(z) \sqrt{|V|}),\quad  z\in \rho(H_0)
$$
is well-defined. Here $\rho(H_0)$ denotes the resolvent set of the operator $H_0$. Furthermore, it can easily be verified that  the perturbation determinant is related to the trace of the resolvent difference by
$$
\frac{D'(z)}{D(z)} = \Tr(R_0(z) - R(z)), \quad z\in \rho(H_0) \cap \rho(H) .
$$
Thus, it follows from \eqref{Tr} that
$$
\frac{D'(z)}{D(z)} =  \frac{1}{2\zeta} \frac{d}{d \zeta} \left(\ln w(\zeta) - \ln (\gamma-i \zeta) \right) = \frac{1}{2\zeta} \frac{d}{d \zeta}\ln \left(\frac{w(\zeta)}{\gamma-i \zeta}\right) = \frac{ \left(\frac{w(\sqrt{z})}{\gamma -i \sqrt{z}}\right)'}{ \left(\frac{w(\sqrt{z})}{\gamma -i \sqrt{z}}\right)}.
$$
Therefore, we conclude that $D(z)= C w(\sqrt{z})/( \gamma - i \sqrt{z})$. Because of the asymptotics \eqref{ruscho}, it follows that
$$
D(z)= \frac{w(\sqrt{z})}{\gamma - i \sqrt{z}} \ .
$$
This is the sought after relation.


\section{Low and High-energy asymptotics}

Here we derive an asymptotic expansion of the perturbation determinant  $D(\zeta)$ as $|\zeta| \to \infty$.

\subsection{High-energy asymptotics}

In this subsection, we assume, that $V\in C^{\infty}(\R_+)$ and that 
\begin{equation}\label{glatt}
|V^{(j)}(x)| \leq C_j (1+x)^{-\rho-j}, \quad \rho \in (1,2],\quad j\in \N_0.
\end{equation}
The asymptotic expansion of the Jost solution $\tet$ for $ | \zeta | \to  \infty$ can be found, e.g., in \cite{Y}. Thereby, it is more convenient to consider the function $b(x,\zeta)$ defined by

\begin{equation}\label{b}
b(x,\zeta) := e^{-ix\zeta} \tet.
\end{equation}
Note, that equation \eqref{eq} for $\tet$ is equivalent to the equation
\begin{equation}\label{eqb}
-b''(x,\zeta) -2i\zeta b'(x,\zeta) + V(x) b(x,\zeta)= 0.
\end{equation}
It follows from \eqref{b} that asymptotics \eqref{asympto} and the asymptotics
\begin{equation}\label{asymptob}
b(x,\zeta) = 1 + o(1), \quad b'(x,\zeta) = o(1), \quad x\to \infty,
\end{equation}
are equivalent to each other. For an arbitrary $N$, the equality
\begin{equation}\label{expansion_b}
b(x,\zeta) = \sum_{n=0}^N b_n(x) (2i\zeta)^{-n} + r_N(x,\zeta)
\end{equation}
holds with the remainder satisfying the estimates
$$
| \partial^j r_N(x,\zeta)/ \partial x^j | \leq C_{N,j} |\zeta|^{-N-1} (1+|x|)^{-(N+1)(\rho -1) -j}, \quad j \in \N_0,
$$
for all $x \geq 0$ and $ \im \zeta \geq 0, \ |\zeta| \geq c >0$. Here $b_0(x)=1$ and $b_n(x)$ are real $C^{\infty}$ functions defined by the recurrent relation 
\begin{equation}\nonumber
b_{n+1} (x) = -b'_n(x) - \int_x^\infty V(y) b_n(y) \,dy.
\end{equation}
Further the following estimates hold,
$$
b_n^{(j)} (x) = O(x^{-n(\rho-1)-j}),\quad j \in \N_0, \quad x \to \infty.
$$
Now, we can prove the asymptotic expansion of the perturbation determinant for $|\zeta| \to \infty$.

\begin{lemma}\label{D}
Suppose $V\in C^{\infty}(\R_+)$ and \eqref{glatt}. Then the perturbation determinant admits the expansion in the asymptotic series
\begin{equation}\label{expansion_pertdet}
D( \zeta) = \sum_{n=0}^\infty d_n (2i\zeta)^{-n}, 
\end{equation}
as $| \zeta | \to \infty, \ \im \zeta \geq 0$. The coefficients $d_n$ are given by
\begin{equation}\label{dn}
d_0 = 1, \quad d_n := b_n(0) + 2 \sum_{m=1}^{n-1} b_m'(0) (2 \gamma)^{n-m-1},\ \ n \geq 1.
\end{equation}
\end{lemma}
We emphasize that expansion \eqref{expansion_pertdet} is understood in the sense of an asymptotic series.
\begin{proof}
It follows from \eqref{expansion_b} that
\begin{equation}\label{anwenden}
\frac{b'(x,\zeta) + (i \zeta - \gamma) b(x,\zeta)}{i \zeta - \gamma} = \frac{1}{i \zeta - \gamma} \sum_{n=0}^\infty b_n'(x) (2 i \zeta)^{-n} + \sum_{n=0}^\infty b_n(x) (2i \zeta)^{-n}.
\end{equation}
Applying the geometric series to the first sum in the right-hand side of \eqref{anwenden} for $|\zeta| > \gamma$, we conclude 
\begin{equation}\label{anwenden2}
\frac{b'(x,\zeta) + (i \zeta - \gamma) b(x,\zeta)}{i \zeta - \gamma} = b_0(x) + \sum_{n=1}^\infty \left( b_n(x) + 2 \sum_{m=0}^{n-1} b_m'(x) (2 \gamma)^{n-m-1}  \right) (2i \zeta)^{-n}.
\end{equation}
On the other hand, it is easy to see that 
\begin{equation}\label{anwenden3}
b'(x,\zeta) + (i \zeta - \gamma) b(x,\zeta) = e^{-i \zeta x} ( \theta'(x,\zeta) - \gamma \tet ) 
\end{equation}
Thus, setting $x=0$ and combining \eqref{anwenden2} with \eqref{anwenden3}, we arrive at \eqref{expansion_pertdet}.
\end{proof}

Note, that because of \eqref{ruscho}, we have for $| \zeta| \to  \infty, \ \ \im \zeta \geq 0$,\\
\begin{equation}\label{1}
D(\zeta) = \frac{w(\zeta)}{\gamma - i\zeta} = 1+ O(| \zeta |^{-1}).
\end{equation}
Thus, we can fix the branch of the the function $\ln D$ by the condition $\ln D(\zeta) \to 0$ as $|\zeta| \to \infty$. The following Corollary is an immediate consequence  of Lemma \ref{D}.
\begin{corollary}\label{lnD}
Suppose $V\in C^{\infty}(\R_+)$ and \eqref{glatt}. Then for $| \zeta | \to \infty, \ \im \zeta \geq 0$, we have
$$
 \ln D(\zeta) = \sum_{n=1}^\infty \ell_n (2i\zeta)^{-n},
 $$
where the coefficients $\ell_n$ are given by
\begin{equation}\label{ln}
\ell_1:= d_1, \quad \ell_n:= d_n - n^{-1} \sum_{j=1}^{n-1} j d_{n-j} \ell_j , \ \ n \geq 2.
\end{equation}
\end{corollary}
The first coefficients $\ell_n$ work out to be

\begin{eqnarray}\nonumber
\ell_1= - \int_0^\infty V(x) \,dx , \quad \ell_2 = V(0),\quad \ell_3 = 4 \gamma V(0) -V'(0) +  \int_0^{\infty}V^2(x) \,dx,
\end{eqnarray}
\begin{eqnarray}\nonumber
\ell_4 = V''(0) - 2 V^2(0) -4 \gamma V'(0) + 8 \gamma^2 V(0).
\end{eqnarray}
From \eqref{nurw} it follows for $k \in \R$, that
$$
\ln D(k)= \ln \left( \frac{w(k)}{\gamma - ik }\right) = \ln a(k) + i\eta(k).
$$
Seperating in Corollary \ref{lnD} the function $\ln D(k)$ into its real and imaginary part, we finally conclude that for $k \to \infty$,
\begin{eqnarray}\label{a}
\ln a(k) &=& \sum_{n=1}^\infty (-1)^n \ell_{2n} (2k)^{-2n},\\ \nonumber 
\eta(k) &=& \sum_{n=0}^\infty (-1)^{n+1} \ell_{2n+1} (2k)^{-2n-1}.
\end{eqnarray}

\subsection{Low-energy asymptotics}
In this section we assume, that
\begin{equation} \label{assu}
\int_0^{\infty} (1+x)|V(x)| \,dx < \infty.
\end{equation} 
We denote the regular solution for $\zeta=0$ by $\varphi(x)$. This is the solution of the integral equation \eqref{ieq},
\begin{equation} \label{varphi(x)}
\varphi(x) = 1 + x\gamma + \int_0^x (x-y) V(y) \varphi(y) \,dy.
\end{equation}
This solution exists under condition \eqref{cond}. As shown in  \cite{Y}, the stronger condition \eqref{assu} guarantees the existence of a Jost solution $\theta(x,\zeta)$ at $\zeta=0$. For any fixed $x \geq 0$, the Jost solution $\theta(x,\zeta)$ is continuous as $\zeta \to 0,\ \im \zeta \geq 0$. Moreover,
$$
| \theta(x,\zeta) - e^{i \zeta x} | \leq e^{- \im \zeta x} \left( \exp\left(C \int_x^{\infty} y |V(y)| \,dy \right)-1 \right),
$$
where $C$ does not depend on $\zeta$ and $x$. The function $\theta(x):= \theta(x,0) = \overline{\theta(x,0)}$ satisfies the equation
\begin{equation}\label{homogen}
-u'' + V(x) u = 0
\end{equation}
and, as $x \to \infty$,
\begin{equation}\label{neueasympto}
\theta(x) = 1+O\left( \int_x^{\infty} y |V(y)| \,dy \right) = 1+o(1), \quad \theta'(x) = O\left(  \int_x^{\infty}  |V(y)| \,dy \right) = o(x^{-1}).
\end{equation}
Indeed, asymptotics \eqref{neueasympto} follow from the integral equation \eqref{inteq} for $\zeta=0$, namely,
$$
\theta(x) = 1 + \int_x^{\infty} (y-x) V(y) \theta(y) \,dy.
$$
One can also show that the Jost function $w(\zeta)$ is continuous as $\zeta \to 0,\ \im \zeta \geq 0$, and from  \eqref{represantation_w} we get for $\zeta = 0$
\begin{equation}\label{representation_w0}
w(0) = \gamma + \int_0^{\infty} V(y) \varphi(y) \,dy.
\end{equation}
If \eqref{assu} holds, then the integral in \eqref{representation_w0} is convergent, in view of the estimate 
$$
| \varphi(x,0) - 1| \leq \tilde{c} \gamma x,
$$
following from \eqref{unique}. Moreover, we have $w(0) = \overline{w(0)}$.

After these preliminaries we claim that the operator $H$ has no zero eigenvalue. Indeed, the function defined by
$$
\tau(x)= \theta(x) \int_{x_0}^x \theta(y)^{-2} \,dy, \quad x \geq x_0,
$$
is a solution of equation \eqref{homogen} and is linearly independent of $\theta(x)$. Again, $x_0$ is an arbitrary point such that $\theta(x) \neq 0$ for $x \geq x_0$. Further,
$$
\tau(x) = x + o(x), \quad \tau'(x) = 1+ o(1) \quad \mbox{as} \ x \to \infty \ \mbox{and} \ W\{\theta,\tau\} = -1.
$$
Thus, the equation \eqref{homogen} does not have solutions, tending to zero at infinity, as claimed.

While the operator does not have a zero eigenvalue, it may have a so-called zero resonance.

\begin{definition}
Under assumption \eqref{assu}, one says that the operator $H$ has a resonance at $\zeta=0$ if $w(0)=0$.
\end{definition}

Since the Jost function is the Wronskian of the Jost and the regular solution, the resonance condition means that $\varphi$ is a multiple of $\theta$ and therefore that equation \eqref{homogen} has a bounded solution satisfying boundary condition \eqref{bc2}.

Now, we want to analyze the behavior of the Jost function $w(\zeta)$ as $\zeta \to 0$. More precisely, we want to show that if $w(0)=0$, then it vanishes not faster than linearly. In order to prove this, we need the following technical Lemma.

\begin{lemma} \label{diff}
Assume \eqref{assu} and let $w(0)=0$. Then 
\begin{equation}\label{varphi-varphi(x)}
|\varphi(x,\zeta) - \varphi(x)| \leq C |\zeta| x e^{\im \zeta x}, \quad \im \zeta \geq 0.
\end{equation}
\end{lemma}

\begin{proof}
We set $\Omega(x,\zeta) = \varphi(x,\zeta) -\varphi(x)$ and 
\begin{equation}\label{Omega_0}
\Omega_0(x,\zeta) = \cos(\zeta x) + \frac{\gamma}{\zeta} \sin(\zeta x) -1- x \gamma -\int_0^x (x-y) V(y) \varphi(y) \,dy + \int_0^x \zeta^{-1} \sin(\zeta (x-y)) V(y) \varphi(y) \,dy.
\end{equation}
It follows from \eqref{ieq} and \eqref{varphi(x)}, that
$$
\Omega(x,\zeta) = \Omega_0(x,\zeta) + \int_0^x \zeta^{-1} \sin(\zeta(x-y)) V(y) \Omega(y,\zeta) \,dy.
$$
We first prove, that
\begin{equation} \label{ersteabschatzung}
|\Omega_0(x,\zeta)| \leq C |\zeta| x e^{\im \zeta x}.
\end{equation}
Note, that the condition $w(0)=0$ is equivalent to the condition $ \int_0^{\infty} V(y) \varphi(y) \,dy  = -\gamma$. Therefore, we can rewrite \eqref{Omega_0} as 
\begin{equation}\label{Omega0}
\Omega_0(x,\zeta)= \cos(\zeta x) - 1 -(\zeta^{-1} \sin(\zeta x) -x ) \int_x^{\infty} V(y) \varphi(y) \,dy + \int_0^x K(x,y,\zeta) V(y) \varphi(y) \,dy,
\end{equation}
where
$$
K(x,y,\zeta) = -\zeta^{-1} \sin(\zeta x) +y+\zeta^{-1} \sin(\zeta(x-y)).
$$
The third and fourth term in the right-hand side of \eqref{Omega0} are bounded from above by $C |\zeta| x e^{\im \zeta x}$. This follows in the same way as shown in \cite{Y}, which only uses that $\varphi(x)$ is bounded. It remains to give an estimate for the first and second term in \eqref{Omega0}, which we write as
$$
\cos (\zeta x) -1 = -2 \sin^2(\zeta x /2).
$$
Using the estimates
$$
| \sin( \zeta x/2)| \leq ce^{| \im \zeta | x/2} \quad \mbox{and} \quad | \sin( \zeta x/2)| \leq c |\zeta | x e^{| \im \zeta | x/2},
$$
we get
$$
| \cos (\zeta x) -1 | \leq c  |\zeta | x e^{| \im \zeta | x}.
$$
Thus, we conclude \eqref{ersteabschatzung}. This inequality implies \eqref{varphi-varphi(x)} by Gronwall's Lemma exactly as in \cite[Lemma 4.3.6]{Y}.
\end{proof}

\begin{proposition}
Under the assumption \eqref{assu} and $w(0)=0$ we have the following asymptotics for the Jost function,
\begin{equation}\label{asymptojostfct}
w(\zeta) = -i w_0 \zeta + o(\zeta), \quad \zeta \to 0,
\end{equation}
where $w_0 = 1-  \int_0^{\infty} yV(y) \varphi(y) \,dy\neq 0$.
\end{proposition}

\begin{proof}
Since $w(0)=0$, we  have $ \int_0^{\infty} V(y) \varphi(y) \,dy  = -\gamma$ and therefore it follows from representation \eqref{represantation_w} that
\begin{eqnarray}\nonumber
w(\zeta) = \int_0^{\infty} \left( e^{i \zeta x} \varphi(x,\zeta) - \varphi(x) \right) V(x) \,dx - i \zeta= &-&i w_0 \zeta + \int_0^{\infty} \left( e^{i\zeta x} -1 -i\zeta x \right) \varphi(x) V(x) \,dx \\ \label{righthandw}
&+&  \int_0^{\infty} e^{i\zeta x} \left(\varphi(x,\zeta)-\varphi(x) \right) V(x) \,dx.
\end{eqnarray}
It can be shown that both integrals in the right-hand side of \eqref{righthandw} are $o(\zeta)$ as $\zeta \to 0$.
Since the function  $e^{i\zeta x} -1 -i\zeta x$ is bounded by $C|\zeta| x$ and is $O(|\zeta|^2)$ for all fixed $x$, it follows that the first integral in the right-hand side is $o(\zeta)$ as $ \varphi(x)$ is a bounded function. The second integral in the right-hand side of \eqref{righthandw} is also $o(\zeta)$. Indeed it follows from Lemma \ref{diff} that the function $e^{i\zeta x} \left(\varphi(x,\zeta)-\varphi(x) \right)$ is bounded by $C|\zeta| x$ and further it is $O(|\zeta|^2)$ for all fixed $x$. Thus the asymptotics \eqref{asymptojostfct} holds.

In order to prove that $w_0\neq 0$, we use equation \eqref{varphi(x)} to write
\begin{align*}
\varphi(x)   
& = x \left(\gamma + \int_0^\infty V(y) \varphi(y)\,dy \right) + w_0 - x \int_x^\infty V(y) \varphi(y)\,dy +  \int_x^\infty y V(y) \varphi(y)\,dy \\
 &= w_0 +o(1)
\end{align*}
as $x\to \infty$. On the other hand, $\varphi$ is proportional to $\theta$, which satisfies \eqref{neueasympto}. This shows that $w_0\neq 0$, as claimed.
\end{proof}


\section{Trace identities}

We now put the material from the previous sections together to prove our main result, namely, a family of trace formulas for the operator $H$. These identities provide a relation between the shift of the spectra between $H$ and $H_0$ and quantities involving the potential $V$. The spectral shift consists of two parts, one coming from the discrete spectrum (expressed in terms of the eigenvalues of $H$ and $H_0$) and the other one coming from the continuous spectrum (expressed in terms of the quantities $\eta$ and $a$). 

In this section we assume that $\int_0^{\infty} (1+x) |V(x)| \,dx < \infty$, which guarantees that $H$ has only a finite number $N$ of negative eigenvalues $\lambda_1,\ldots,\lambda_N$. We recall that $H_0$ has a single negative eigenvalue $-\gamma^2$ if $\gamma<0$ and no negative eigenvalues if $\gamma\geq 0$. We also recall that $M_s(\gamma)$ was defined at the end of the introduction.

While we are mainly interested in trace formulas of integer and half-integer order, we prove a version of these formulas for every complex $s$ with $\re s>0$. We proceed by analytic continuation, where the starting point is the following proposition.

\begin{proposition}
Suppose that $\int_0^{\infty} (1+x) |V(x)| \,dx < \infty$, $s\in \C$, $0< \re s < 1/2$ and define
\begin{equation}\label{F,G}
F(s):= \int_0^{\infty} \ln a(k) k^{2s-1} \,dk, \quad G(s):= \int_0^{\infty} \eta(k) k^{2s-1} \,dk.
\end{equation}
Then 
\begin{equation}\label{trformula}
\frac{\pi}{2s} \sum_{j=1}^N |\lambda_j|^s  - \frac{ \pi}{2s} M_s(\gamma) = \sin(\pi s) F(s) - \cos(\pi s) G(s).
\end{equation}
\end{proposition}
\begin{proof}
Let $\Gamma_{R,\varepsilon}$ be the contour (with counterclockwise direction) which consists of the halfcircles $C_R^+ = \{ |\zeta| =R, \im \zeta \geq 0 \}$ and $C_{\varepsilon}^+ = \{ |\zeta| = \varepsilon , \ \im \zeta \geq 0 \}$ and the intervals $(\varepsilon, R)$ and $(-R,-\varepsilon)$. 
\vspace*{4cm}
\begin{center}
\begin{figure}[h!]
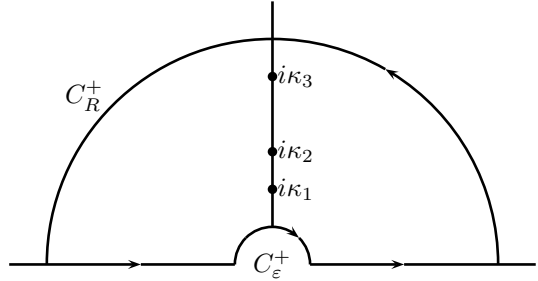

\psset{unit=0.5cm}
\psarc[linewidth=1pt]{-}(0,0){1}{0}{45}
\psarc[linewidth=1pt]{<-}(0,0){1}{45}{135}
\psarc[linewidth=1pt](0,0){1}{135}{180}
\psarc[linewidth=1pt]{->}(0,0){6}{0}{60}
\psarc[linewidth=1pt](0,0){6}{60}{180}
\psline[linewidth=1pt]{->}(1,0)(3.5,0)
\psline[linewidth=1pt]{-}(3.5,0)(7,0)
\psline[linewidth=1pt]{->}(-7,0)(-3.5,0)
\psline[linewidth=1pt]{-}(-3.5,0)(-1,0)
\psline[linewidth=1pt]{-}(0,1)(0,7)
\psdot(0,3)
\psdot(0,2)
\psdot(0,5)
\rput(0.5,2){$\ i\kappa_1$}
\rput(0.5,3){$\ i\kappa_2$}
\rput(0.5,5){$\ i\kappa_3$}
\rput(-5,4.5){$C_R^+$}
\rput(0,0){$C_{\varepsilon}^+$}
\caption{contour of integration}
\end{figure}
\end{center}
The argument of $\zeta \in \C$ is fixed by the condition  $0 \le \arg \zeta \leq \pi$. 
We consider the integral
\begin{equation}\nonumber
\int_{\Gamma_{R, \varepsilon}} \frac{\frac{d}{d\zeta}\left( \frac{w(\zeta)}{\gamma -i\zeta} \right)}{\left( \frac{w(\zeta)}{\gamma -i\zeta} \right)} \zeta^{2s} \,d\zeta.
\end{equation}
The set of the singularities of the integrand is the set of zeros of the function $w(\zeta)/(\gamma - i\zeta)$.

Calculating the integral by residues, we see that for $\kappa_j = |\lambda_j|^{1/2}$
\begin{equation}\label{Res}
\int_{\Gamma_{R,\varepsilon}} \frac{\frac{d}{d\zeta}\left( \frac{w(\zeta)}{\gamma - i\zeta} \right)}{\left( \frac{w(\zeta)}{\gamma - i\zeta} \right)} \zeta^{2s} \,d\zeta
=
2 \pi i \sum_{j=1}^N \mbox{Res}_{\zeta=i \kappa_j} \frac{\frac{d}{d\zeta}\left( \frac{w(\zeta)}{\gamma - i\zeta} \right)}{\left( \frac{w(\zeta)}{\gamma-i\zeta} \right)} \zeta^{2s} - 2\pi i \begin{cases} \mbox{Res}_{\zeta=-i \gamma} \frac{\frac{d}{d\zeta}\left( \frac{w(\zeta)}{\gamma - i\zeta} \right)}{\left( \frac{w(\zeta)}{\gamma-i\zeta} \right)} \zeta^{2s},  &\mbox{if}\  \gamma < 0 \\ 0,\   &\mbox{if} \  \gamma \geq 0.\end{cases}
\end{equation}
Since by Lemma \ref{zero}, zeros $i \kappa_j$ of $w(\zeta)$ are simple, the residues work out to be $e^{i\pi s} \kappa_j^{2s}$. Hence,
\begin{equation}\label{istanbul}
\int_{\Gamma_{R,\varepsilon}} \frac{\frac{d}{d\zeta}\left( \frac{w(\zeta)}{\gamma - i\zeta} \right)}{\left( \frac{w(\zeta)}{\gamma - i\zeta} \right)} \zeta^{2s} \,d\zeta
=
2 \pi i \ e^{i \pi s} \sum_{j=1}^N  \kappa_j^{2s} - 2 \pi i \ e^{i \pi  s} M_s(\gamma).
\end{equation}
Next, we show that the integral over the semicircle $C_r^+$ tends to zero as $r \to \infty$ or $ r \to 0$. Integrating by parts, we see that
\begin{eqnarray}\nonumber
\int_{C_r^+}  \frac{\frac{d}{d\zeta}\left( \frac{w(\zeta)}{\gamma - i\zeta} \right)}{\left( \frac{w(\zeta)}{\gamma - i\zeta} \right)} \zeta^{2s} \,d\zeta = -2s \int_{C_r^+} \ln \left( \frac{w(\zeta)}{\gamma - i\zeta} \right) \zeta^{2s-1} \,d \zeta + \ln \left( \frac{w(-r)}{\gamma + ir} \right)(-r)^{2s} - \ln \left( \frac{w(r)}{\gamma - ir} \right)r^{2s} .
\end{eqnarray}
Note, that we can choose $\ln (w(\zeta)/ (\gamma - i\zeta))$ as a continuous function on $C^+_r$. If $r \to \infty$, then this integral tends to zero for $\re \ s < 1/2$ because of \eqref{1}. If $r \to 0$, then the integral also tends to zero. Indeed, this follows from the fact that either $w(0) \neq 0$ or $w(\zeta)$ satisfies \eqref{asymptojostfct} with $w_0 \neq 0$. 
Therefore passing to the limits $R\to \infty$ and $\varepsilon \to 0$ in equality  \eqref{istanbul}, we obtain that

\begin{equation}\label{ev}
\int_{-\infty}^{\infty} \frac{\frac{d}{dk}\left( \frac{w(k)}{\gamma - ik} \right)}{\left( \frac{w(k)}{\gamma - ik} \right)} k^{2s} \,dk
=
2 \pi i \ e^{i \pi s} \sum_{j=1}^N  \kappa_j^{2s} - 2 \pi i \ e^{i \pi s}\begin{cases} (-\gamma)^{2s}, \ \  \mbox{if}\  \gamma < 0 \\ 0,\ \quad \quad \quad  \mbox{if}\  \gamma \geq 0. \end{cases}
\end{equation}
Integrating in the left-hand side by parts and taking into account relations \eqref{realw} and \eqref{nurw}, we obtain
\begin{eqnarray}\nonumber
\int_{-\infty}^{\infty}  \frac{\frac{d}{dk}\left( \frac{w(k)}{\gamma - ik} \right)}{\left( \frac{w(k)}{\gamma - ik} \right)} k^{2s} \,dk
&=& -2s \int_{-\infty}^{\infty} \ln \left(\frac{w(k)}{\gamma - ik} \right) k^{2s-1} \,dk
= -2s  \int_{-\infty}^{\infty} (\ln a(k) +i \eta(k))k^{2s-1} \,dk \\ \nonumber
&=& -2s  \int_{0}^{\infty} (\ln a(k) +i \eta(k))k^{2s-1} \,dk +  2s e^{2i \pi s} \int_{0}^{\infty} (\ln a(k) - i \eta(k))k^{2s-1} \,dk \\ \nonumber
&=& 2s( e^{2i \pi s}-1)F(s) -2is ( e^{2i \pi s} + 1)G(s).
\end{eqnarray}
Comparing this equation with \eqref{ev}, we arrive at \eqref{trformula}.
\end{proof}
In order to prove trace identities for arbitrary powers $s\in \C_+$, we need the analytic continuation of the functions $F(s)$ and $G(s)$ to the entire half-plane $\re s >0$. 

\begin{lemma}\label{wichtig}
Let estimates \eqref{glatt} and $\int_0^{\infty} (1+x) |V(x)| \,dx < \infty$ be satisfied. Then the functions $F$ and $G$ are meromorphic in the half-plane $ \re s>0$. The function $F$ is analytic everywhere except for simple poles at integer points $s=n, \ n \in \N,$ with residues
$$
\mathrm{Res}_{s=n} F(s) = (-1)^{n+1} 2^{-2n-1} \ell_{2n}, \quad n \in \N.
$$
If $\re s <1$, then representation \eqref{F,G} for $F(s)$ remains true. If $ n< \re s< n+1$, then
\begin{equation}\nonumber
F(s)= \int_0^{\infty} \left( \ln a(k) - \sum_{j=1}^n (-1)^j \ell_{2j} (2k)^{-2j} \right) k^{2s-1} \,dk.
\end{equation}
The function $G$ is analytic everywhere except for simple poles at half-integer points $s=n+1/2,\  n\in \N_0,$ with residues 
$$
\mathrm{Res}_{s=n+1/2} G(s) =  (-1)^n 2^{-2n-2} \ell_{2n+1}, \quad n \in \N_0.
$$
\\
If $n \geq 1$ and $n-1/2 < \re s < n+1/2$, then
\begin{equation}\label{Gfort}
G(s) = \int_0^{\infty} \left( \eta(k) - \sum_{j=0}^{n-1} (-1)^{j+1} \ell_{2j+1} (2k)^{-2j-1} \right) k^{2s-1} \,dk.
\end{equation}
\end{lemma}

\begin{proof}
We can write the function $F$, given in Lemma \ref{wichtig} as follows
\begin{eqnarray}\label{Fbeweis}
F(s) = \int_0^1 \ln a(k) k^{2s-1} \,dk +  \int_1^{\infty} \left( \ln a(k) -  \sum_{j=1}^n (-1)^j \ell_{2j} (2k)^{-2j} \right) k^{2s-1} \,dk \\ \nonumber
- \int_0^1 \sum_{j=1}^n (-1)^j \ell_{2j} (2k)^{-2j} k^{2s-1} \,dk.
\end{eqnarray}
The first integral in the right-hand side of equation \eqref{Fbeweis} is an analytic function of $s$ in the entire half-plane $\re s>0$. The second integral is in view of \eqref{a} an analytic function of $s$ in the strip $0< \re s <n+1$. 
For $\re s >n$, we have
$$
\int_0^1 \sum_{j=1}^n (-1)^j \ell_{2j} (2k)^{-2j} k^{2s-1} \,dk = \sum_{j=1}^n (-1)^{j} 2^{-2j-1} \ell_{2j} (s-j)^{-1}.
$$
Thus, the function $F$ is an analytic function in the strip $n<\re s < n+1$.\\Similarly, we split the integral in the right-hand side of \eqref{Gfort}. Note, that we now have for $\ n \geq 1$ and $\re s > n-1/2$,
\begin{equation}\nonumber 
 \int_0^1  \sum_{j=0}^{n-1} (-1)^{j+1} \ell_{2j+1} (2k)^{-2j-1}  k^{2s-1} \,dk =  \sum_{j=0}^{n-1} (-1)^{j+1} \ell_{2j+1} 2^{-2j-2} (s-j-1/2)^{-1}.
 \end{equation}
Therefore, it follows with analog arguments  as for $F$ that the function $G$, given in \eqref{Gfort}, is an analytic function of $s$ in the strip $n-1/2 < \re s <n+1/2$.
\end{proof}

\begin{theorem}
Let estimates \eqref{glatt} and $\int_0^{\infty} (1+x) |V(x)| \,dx < \infty$ be satisfied. Then
\begin{eqnarray}\label{neu1}
\sum_{j=1}^N |\lambda_j|^{1/2} - M_{1/2}(\gamma)- \frac{1}{\pi} \int_0^{\infty} \ln a(k) \,dk = \frac{1}{4} \ell_1
\end{eqnarray}
and for $n \geq 1,\ n \in \N$,
\begin{equation}\label{neu2}
\sum_{j=1}^N  |\lambda_j|^n - M_n(\gamma) + (-1)^n \frac{2n}{\pi}  \int_0^{\infty} \left( \eta(k) - \sum_{j=0}^{n-1} (-1)^{j+1} \ell_{2j+1} (2k)^{-2j-1} \right) k^{2n-1} \,dk = -\frac{n}{2^{2n}} \ell_{2n},
\end{equation}
\begin{eqnarray}\nonumber
\sum_{j=1}^N |\lambda_j|^{n+1/2} - M_{n+1/2}(\gamma) + (-1)^{n+1} \frac{2n+1}{\pi} \int_0^{\infty} \left( \ln a(k) - \sum_{j=1}^n (-1)^j \ell_{2j} (2k)^{-2j} \right) k^{2n} \,dk \\ \label{neu3}
= \frac{2n+1}{2^{2n+2}} \ell_{2n+1}.
\end{eqnarray}
The coefficients $\ell_n$ are given as in \eqref{ln}.
\end{theorem}

\begin{proof}
Using the analytic continuation, given in Lemma \ref{wichtig}, formula \eqref{trformula} can be extended to all $s$ in the half-plane $\re s >0$. In particular, setting $s=n,\ n \in \N$, we obtain
$$
 \frac{\pi}{2n} \sum_{j=1}^N |\lambda_j|^n -  \frac{\pi}{2n}  M_n(\gamma) = (-1)^n \left( \pi \ \mbox{Res}_{s=n} F(s)  - G(n) \right).  
$$
Taking into account Lemma \ref{wichtig}, we conclude formula \eqref{neu2}. Similarly, we have for $s= n+1/2$, where $n\in \N_0$, the following identity
$$
\frac{\pi}{2n+1}  \sum_{j=1}^N |\lambda_j|^{n+1/2} -  \frac{\pi}{2n+1} M_{n+1/2}(\gamma)
= (-1)^n \left( F(n+1/2) + \pi  \ \mbox{Res}_{s=n+1/2} G(s) \right).
$$
Again, in view of Lemma \ref{wichtig}, we conclude formulas \eqref{neu1} and \eqref{neu3}.
\end{proof}

We compute the first four trace formulas.
\begin{corollary}
Let estimates \eqref{glatt} and $\int_0^{\infty} (1+x) |V(x)| \,dx < \infty$ be satisfied. Then
\begin{eqnarray}\nonumber 
\sum_{j=1}^N |\lambda_j|^{1/2} - M_{1/2}(\gamma)- \frac{1}{\pi} \int_0^{\infty} \ln a(k) \,dk &=&- \frac{1}{4} \int_0^{\infty} V(x) \,dx, \\ \nonumber
\sum_{j=1}^N |\lambda_j| - M_1(\gamma)- \frac{2}{\pi}  \int_0^{\infty} \left( \eta(k) -\frac{\int_0^{\infty}V(x)\,dx}{2k} \right) k \,dk &=& -\frac 14 \,V(0), \\ \nonumber
\sum_{j=1}^N |\lambda_j|^{3/2} - M_{3/2}(\gamma) +  \frac{3}{\pi}  \int_0^{\infty} \left( \ln a(k) + \frac{V(0)}{(2k)^2} \right) k^2 \,dk &=& \frac{3}{16} \left(\int_0^{\infty}V^2(x)\,dx-V'(0) +  4 \gamma V(0) \right),
\end{eqnarray}
\begin{eqnarray}\nonumber
\sum_{j=1}^N |\lambda_j|^2 - M_2(\gamma)+ \frac{4}{\pi}  \int_0^{\infty} \left( \eta(k) -\frac{\int_0^{\infty}V(x)\,dx}{2k} - \frac{\int_0^{\infty} V^2(x) \,dx + 4 \gamma V(0) -V'(0)}{(2k)^3} \right) k^3 \,dk= - \frac18 \ell_4,
\end{eqnarray}
where 
$$
\ell_4 =V''(0) +8 \gamma^2 V(0) -4 \gamma V'(0) - 2V^2(0).
$$
\end{corollary}

Finally, we prove a trace formula of order zero for the operator $H$ with boundary conditions \eqref{bc2}. Such formulas are called in the literature the \textit{Levinson formula} and 
relate the number of negative eigenvalues of $H$ to the phase shift $\eta$. 

We define $\eta(0) = \lim_{\zeta \to 0+} \eta(\zeta)$. Obviously this limit exists if $w(0) \neq 0$. In the case $w(0)=0$ the existence follows from asymptotics \eqref{asymptojostfct}.

\begin{theorem}\label{levinson}
Suppose \eqref{assu} and let $N$ be the number of negative eigenvalues of the operator $H$ with boundary condition \eqref{bc2}. Then, the following formulas hold.\\For $w(0) \neq 0$,
\begin{equation}\label{ungleich0}
\eta(\infty) - \eta(0) = \begin{cases} \pi N \ \  &\mbox{if}\  \gamma > 0, \\ \pi (N-\frac 12) \ \quad  &\mbox{if}\  \gamma = 0, \\ \pi (N-1) \ \quad  &\mbox{if}\  \gamma <  0 . \end{cases}
\end{equation}
For $w(0)=0$,
\begin{equation}\label{gleich0}
\eta(\infty)- \eta(0) =  \begin{cases} \pi (N + \frac 12)  \ \  &\mbox{if}\  \gamma > 0,  \\ \pi N \ \quad  &\mbox{if}\  \gamma =  0 ,  \\  \pi (N - \frac{1}{2}) \ \quad  &\mbox{if}\  \gamma <  0. \end{cases}
\end{equation}
\end{theorem}

\begin{proof}
We apply the argument principle to the function $D(\zeta)$ and the contour $\Gamma_{R,\varepsilon}$ given in Figure $1$. We choose $R$ and $\varepsilon$ such that all of the $N$ negative eigenvalues of $H$ lie inside the contour $\Gamma_{R,\varepsilon}$. Remember that if $\gamma <0$, then $H_0$ has a simple negative eigenvalue $-\gamma^2$. Thus, it follows with Lemma \ref{zero} that
\begin{equation}\label{argumentprinciple}
\int_{\Gamma_{R,\varepsilon}} \frac{\frac{d}{d\zeta} \left(\frac{w(\zeta)}{\gamma-i\zeta} \right)}{ \left(\frac{w(\zeta)}{\gamma-i\zeta} \right)} \,d\zeta =  \begin{cases} 2 \pi i N  \ \  &\mbox{if}\  \gamma \geq 0,  \\ 2 \pi i (N-1) \ \quad  &\mbox{if}\  \gamma <  0. \end{cases}
\end{equation}
Note, that the branch of the function $\ln D(\zeta)$ was fixed by the condition $\ln D(\zeta) \to 0$ as $|\zeta| \to \infty$. Thus, we have $ \ln D(\zeta) = \ln |D(\zeta)| +i \mbox{arg}\ D(\zeta)$. As for $k \in \R$, $\mbox{arg} \ D(k) = \eta(k)$ and $ \eta(-k) = -\eta(k)$, it follows from equation \eqref{argumentprinciple} that
\begin{equation}\label{combine}
2 (\eta(R) - \eta(\varepsilon)) + \mbox{var}_{C^+_R} \mbox{arg}\ D(\zeta)+  \mbox{var}_{C^+_{\varepsilon}} \mbox{arg}\ D(\zeta) = \begin{cases} 2 \pi N \ \  &\mbox{if}\  \gamma \geq 0,  \\ 2 \pi (N-1) \ \quad  &\mbox{if}\  \gamma <  0 . \end{cases}
\end{equation}
Note, that $\lim_{R \to \infty} \mbox{var}_{C^+_R} \mbox{arg}\ D(\zeta) = 0$ because of \eqref{1}. To compute $ \lim_{\varepsilon \to 0} \mbox{var}_{C^+_{\varepsilon}} \mbox{arg} \ D(\zeta)$, we rewrite
\begin{equation}\label{differenceende}
 \mbox{var}_{C^+_{\varepsilon}} \mbox{arg} \ D(\zeta) =  \mbox{var}_{C^+_{\varepsilon}} \mbox{arg} \ w(\zeta) -  \mbox{var}_{C^+_{\varepsilon}} \mbox{arg}(\gamma -i \zeta).
 \end{equation}
Considering \eqref{asymptojostfct}, we get
\begin{equation}\label{levi1}
\lim_{\varepsilon \to 0} \mbox{var}_{C^+_{\varepsilon}} \mbox{arg}\ w(\zeta )= \begin{cases} 0 \ \  &\mbox{if}\  w(0) \neq 0,  \\ -\pi \ \quad  &\mbox{if}\  w(0) =  0. \end{cases} 
\end{equation}
The second term in the right-hand side of \eqref{differenceende} depends on the sign of $\gamma$ and turns out to be in the limit,
\begin{equation}\label{levi2}
\lim_{\varepsilon \to 0} \mbox{var}_{C^+_{\varepsilon}} \mbox{arg} (\gamma -i \zeta )= \begin{cases} 0 \ \  &\mbox{if}\  \gamma > 0,  \\ -\pi \ \quad  &\mbox{if}\  \gamma =  0 ,  \\  0 \ \quad  &\mbox{if}\  \gamma <  0. \end{cases} 
\end{equation}
Combining \eqref{levi1} and \eqref{levi2} with equation \eqref{combine}, formulas \eqref{ungleich0} and \eqref{gleich0} follow immediately.
\end{proof}


\appendix
\section{Proof of Lemma \ref{appendix} }
\noindent For the sake of completeness, we prove here Lemma \ref{appendix}. In order to do so, we have to recall the construction of the Jost solution from \cite[Lemma 4.1.4]{Y} .

Instead of constructing $\tet$, it is more convenient to construct the function $b(x,\zeta)$, defined in \eqref{b} with asymptotics \eqref{asymptob}.
Then differential equation \eqref{eqb} is equivalent to the integral equation
\begin{equation}\label{inteq}
b(x,\zeta) = 1+ (2i\zeta)^{-1} \int_x^{\infty} ( e^{2i\zeta (y-x)} - 1)V(y) b(y,\zeta) \ dy
\end{equation}
considered on the class of bounded functions $b(x,\zeta)$. Its solution $b(x,\zeta)$ can be constructed by the method of successive approximations. Set $b_0(x,\zeta) = 1$,
\begin{equation}\label{bn+1}
b_{n+1}(x,\zeta)= (2i\zeta)^{-1} \int_x^{\infty} ( e^{2i\zeta (y-x)} - 1)V(y) b_n(y,\zeta) \ dy
\end{equation}
and use an obvious estimate
\begin{equation}\nonumber 
| e^{2i\zeta(y-x)} - 1| \leq 2, \quad x \leq y,\  \im \zeta \geq 0.
\end{equation}
Under assumption $\eqref{cond}$, it follows successively that 
\begin{equation}\label{bn}
| b_n(x,\zeta) | \leq  | \zeta |^{-n} (n!)^{-1} \left( \int_x^{\infty} |V(y)|  \ dy \right)^n.
\end{equation} 
For any fixed $x \geq 0$, every function $b_n(x,\zeta)$ is analytic in $\zeta$ in the upper half-plane $\im \zeta > 0$ and is continuous in $\zeta$ up to the real axis, with exception of the point $\zeta = 0$. It follows from \eqref{bn} that the limit 
\begin{equation}\label{limit}
b(x,\zeta) :=  \lim_{N \to \infty} b^{(N)} (x,\zeta), \quad \mbox{where} \ \ b^{(N)} (x,\zeta) = \sum_{n=0}^N b_n(x,\zeta),
\end{equation}
exists for all $x \geq 0$ uniformly with respect to $\zeta$ for $| \zeta | \geq c > 0$. Therefore the function $b(x,\zeta)$ has the same analytic properties in the variable $\zeta$ as the functions $b_n(x, \zeta)$. Moreover, estimates \eqref{bn} show that $b(x,\zeta)$ is a bounded function of $x$. Putting together definitions  \eqref{bn+1} and \eqref{limit}, we see that
$$
b^{(N)} (x,\zeta) = 1-b_{N+1}(x,\zeta) + (2i\zeta)^{-1} \int_x^{\infty} ( e^{2i\zeta (y-x)} - 1)V(y) b^{(N)}(y,\zeta) \ dy.
$$
Passing here to the limit $N \to \infty$, we arrive at equation \eqref{inteq}. The uniqueness of a bounded solution $b(x,\zeta)$ of equation \eqref{inteq} follows from estimate \eqref{bn} for a solution of the corresponding homogeneous equation. This estimate implies that $b(x,\zeta)=0$.

\begin{proof}[Proof of Lemma \ref{appendix}]
By \eqref{limit}, we have
\begin{equation}\label{darstellungtet}
\tet = e^{i\zeta x} \lim_{N\to \infty} \sum_{n=0}^N b_n(x,\zeta),
\end{equation}
where $b_n(x,\zeta)$ is given as in \eqref{bn+1}. As the limit in \eqref{darstellungtet} exists for all $x\geq 0$ uniformly with respect to $\zeta$ for $|\zeta| \geq c >0$, it follows that
\begin{equation}\nonumber
\theta'(x,\zeta) = i\zeta e^{i\zeta x} \lim_{N\to \infty} \sum_{n=0}^N b_n(x,\zeta) + e^{i\zeta x} \lim_{N\to \infty} \sum_{n=0}^N b'_n(x,\zeta).
\end{equation}
Obviously, we have $b'_0(x,\zeta) = 0$ and from \eqref{bn+1}, we derive 
$$
b'_{n}(x,\zeta)= - \int_x^{\infty} e^{2i\zeta (y-x)} V(y) b_{n-1}(y,\zeta) \ dy.
$$
Thus,
\begin{eqnarray}\label{semra}\nonumber
\left| \theta'(x,\zeta) -i \zeta e^{i\zeta x} \right| &=& e^{-\im \zeta x} |\zeta| \left| \lim_{N\to \infty} \sum_{n=1}^N b_n(x,\zeta) + \frac{1}{i\zeta} b'_n(x,\zeta) \right| \\
&=& \frac{e^{-\im \zeta x}}{2} \left| \lim_{N\to \infty} \sum_{n=1}^N \int_x^{\infty} \left(e^{2i\zeta(y-x)} +1 \right) V(y) b_{n-1}(y,\zeta) \,dy \right| .
\end{eqnarray}
Using the estimate
$$
|e^{2i\zeta(y-x)} +1| \leq 2, \quad x\leq y,\quad \im \zeta \geq 0,
$$
we see that the term in \eqref{semra} does not exceed
\begin{equation}\label{semra1}
e^{-\im \zeta x}  \lim_{N\to \infty} \sum_{n=0}^N \int_x^{\infty} \left|V(y)\right| |b_{n}(y,\zeta)| \,dy.
\end{equation}
Because of \eqref{bn}, expression \eqref{semra1} is less or equal
$$
e^{-\im \zeta x}  \lim_{N\to \infty} \sum_{n=0}^N |\zeta|^{-n} (n!)^{-1} \int_x^{\infty} \left|V(y)\right| \left(\int_y^{\infty} |V(t)| \,dt  \right)^n \,dy,
$$
which is equivalent to 
$$
e^{-\im \zeta x} |\zeta| \left( \exp\left( |\zeta|^{-1} \int_x^{\infty} |V(y)| \,dy \right) - 1 \right).
$$
This proves \eqref{eigeneabschatzung}. If $|\zeta| \geq k>0$ and condition \eqref{cond} is satisfied, then estimate \eqref{schluss} follows from \eqref{eigeneabschatzung}.
\end{proof}

\bibliographystyle{plain}
\bibliography{dh.bib}

\def\cprime{$'$} \def\cprime{$'$}
\begin{thebibliography}{10}

\bibitem{A}
Tuncay Aktosun.
\newblock Inverse scattering on the line with incomplete scattering data.
\newblock In {\em Partial differential equations and inverse problems}, volume
  362 of {\em Contemp. Math.}, pages 1--11. Amer. Math. Soc., Providence, RI,
  2004.

\bibitem{BR}
S.~M. Belov and A.~V. Rybkin.
\newblock Higher order trace formulas of the {B}uslaev-{F}addeev-type for the
  half-line {S}chr\"odinger operator with long-range potentials.
\newblock {\em J. Math. Phys.}, 44(7):2748--2761, 2003.

\bibitem{BY}
M.~S. Birman and D.~R. Yafaev.
\newblock The spectral shift function: The papers of {M}. {G}. {K}rein and
  their further developments.
\newblock {\em St. Petersburg Math. J.}

\bibitem{B}
D.~Boll\'e.
\newblock Sum rules in scattering theory and applications to statistical
  mechanics.
\newblock In {\em Lectures on recent results}, volume~2 of {\em
  Mathematics+Physics}, pages 84--153. World Scientific, Singapore, 1986.

\bibitem{BT}
A.~Boumenir and V.~K. Tuan.
\newblock A trace formula and {S}chmincke inequality on the half-line.
\newblock {\em Proc. Amer. Math. Soc.}, 137(3):1039--1049, 2009.

\bibitem{BF}
V.~S. Buslaev and L.~D. Faddeev.
\newblock Formulas for traces for a singular {S}turm-{L}iouville differential
  operator.
\newblock {\em Soviet Math. Dokl.}, 1:451--454, 1960.

\bibitem{DT}
P.~Deift and E.~Trubowitz.
\newblock Inverse scattering on the line.
\newblock {\em Comm. Pure Appl. Math.}, 32(2):121--251, 1979.

\bibitem{D}
L.~A. Diki{\u\i}.
\newblock Trace formulas for {S}turm-{L}iouville differential operators.
\newblock {\em Uspehi Mat. Nauk (N.S.)}, 13(3(81)):111--143, 1958.

\bibitem{GL}
I.~M. Gel{\cprime}fand and B.~M. Levitan.
\newblock On a simple identity for the characteristic values of a differential
  operator of the second order.
\newblock {\em Doklady Akad. Nauk SSSR (N.S.)}, 88:593--596, 1953.

\bibitem{GH}
F.~Gesztesy and H.~Holden.
\newblock On trace formulas for {S}chr\"odinger-type operators.
\newblock In {\em Multiparticle quantum scattering with applications to
  nuclear, atomic and molecular physics ({M}inneapolis, {MN}, 1995)}, volume~89
  of {\em IMA Vol. Math. Appl.}, pages 121--145. Springer, New York, 1997.

\bibitem{GHSZ}
F.~Gesztesy, H.~Holden, B.~Simon, and Z.~Zhao.
\newblock A trace formula for multidimensional {S}chr\"odinger operators.
\newblock {\em J. Funct. Anal.}, 141(2):449--465, 1996.

\bibitem{K}
Rowan Killip.
\newblock Spectral theory via sum rules.
\newblock In {\em Spectral theory and mathematical physics: a Festschrift in
  honor of Barry Simon's 60th birthday}, volume~76 of {\em Proc. Sympos. Pure
  Math.}, pages 907--930.

\bibitem{LW}
Ari Laptev and Timo Weidl.
\newblock Sharp {L}ieb-{T}hirring inequalities in high dimensions.
\newblock {\em Acta Math.}, 184(1):87--111, 2000.

\bibitem{LT}
Elliott~H. Lieb and Walter Thirring.
\newblock Inequalities for the moments of the eigenvalues of the
  {S}chr\"odinger {H}amiltonian and their relation to {S}obolev inequalities.
\newblock {\em Studies in Mathematical Physics, Princeton University Press,
  Princeton, NJ}, pages 269--303, 1976.

\bibitem{R}
Alexei Rybkin.
\newblock On a trace formula of the {B}uslaev-{F}addeev type for a long-range
  potential.
\newblock {\em J. Math. Phys.}, 40(3):1334--1343, 1999.

\bibitem{Ya}
D.~R. Yafaev.
\newblock Trace formulas for charged particles in nonrelativistic quantum
  mechanics.
\newblock {\em Teoret. Mat. Fiz.}, 11(1):78--92, 1972.

\bibitem{Y}
D.~R. Yafaev.
\newblock {\em Mathematical Scattering Theory, Analytic Theory}, volume 158 of
  {\em Mathematical Surveys and Monographs}.
\newblock American Mathematical Society, Providence, RI, 2010.

\bibitem{FZ}
V.~E. Zaharov and L.~D. Faddeev.
\newblock The {K}orteweg-de {V}ries equation is a fully integrable
  {H}amiltonian system.
\newblock {\em Funkcional. Anal. i Prilo\v zen.}, 5(4):18--27, 1971.

\end{thebibliography}
\end{document}